\documentclass[twoside,a4paper,reqno,12pt]{amsart} 
\usepackage[top=28mm,right=30mm,bottom=28mm,left=30mm]{geometry}

\usepackage{amsfonts, amsmath, amssymb, mathrsfs, bm, latexsym, stmaryrd, array, mathtools, bold-extra}

\usepackage{etaremune}
\usepackage{enumitem}

\usepackage[]{hyperref}

\newcommand{\la}{\langle}
\newcommand{\ra}{\rangle}

\newcommand{\leqs}{\leqslant}
\newcommand{\geqs}{\geqslant}

\newcommand{\Aut}{\operatorname{Aut}}
\newcommand{\Out}{\operatorname{Out}}

\newcommand{\GL}{{\mathrm {GL}}}

\newcommand{\SL}{{\mathrm {SL}}}

\newcommand{\SU}{{\mathrm {SU}}}

\newcommand{\FF}{\mathbb{F}} 
\newcommand{\LL}{\mathrm {L}}
\newcommand{\UU}{\mathrm {U}}

\newcommand{\GU}{\mathrm {GU}}

\newcommand{\Alt}{{\mathrm {A}}}
\newcommand{\Sym}{{\mathrm {S}}}
\newcommand{\Syl}{{\mathrm {Syl}}}

\makeatletter
\newcommand{\imod}[1]{\allowbreak\mkern4mu({\operator@font mod}\,\,#1)}
\makeatother

\renewcommand{\leq}{\leqs}
\renewcommand{\geq}{\geqs}

\newtheorem{theorem}{Theorem} 
\newtheorem*{conj*}{Conjecture}

\newtheorem{corol}{Corollary}

\newtheorem{thm}{Theorem}[section] 
\newtheorem{prop}[thm]{Proposition} 
\newtheorem{lem}[thm]{Lemma}

\theoremstyle{definition}
\newtheorem{rem}[thm]{Remark}

\newtheorem{ex}{Example}

\begin{document}

\title[Extremely closed subgroups]{Extremely closed subgroups and a variant on Glauberman's $Z^*$-theorem}

\author{Hung P. Tong-Viet}
\address{H.P. Tong-Viet, Department of Mathematics and Statistics, Binghamton University, Binghamton, NY 13902-6000, USA}
\email{htongvie@binghamton.edu}

\renewcommand{\shortauthors}{Tong-Viet}

\begin{abstract} Let $G$ be a finite group and let $H$ be a subgroup of $G$. We say that $H$ is  extremely closed in $G$ if $\langle H,H^g\rangle\cap N_G(H)=H$ for all $g\in G.$ In this paper, we determine the structure of finite groups with an  extremely closed abelian $p$-subgroup for some prime $p$. In particular, we show that if $G$ contains such a subgroup $H$, then  $G=N_G(H)O_{p'}(G).$  This is a variant on the celebrated Glauberman's $Z^*$-Theorem.
\end{abstract}

\date{\today}


\maketitle

\section{Introduction}
It is an important problem in finite group theory to determine whether a finite group is simple or not. Many non-simplicity criteria have been obtained in the literature.  Among those is the celebrated Glauberman's $Z^*$-theorem. To state this theorem, we need some definitions. Let $G$ be a finite group and let $p$ be a prime. Let $x\in G$ be a $p$-element and let $P$ be a Sylow $p$-subgroup of $G$ containing $x$. We say that $x$ is isolated in $P$ with respect to $G$ if $x^G\cap P=\{x\}$, that is, $x$ is not conjugate in $G$ to any element in $P$ except for $x$ itself. Here, $x^G$ denotes the conjugacy class of $G$ containing $x$. We say that $x$ is isolated in $G$ if $x$ is isolated in some Sylow $p$-subgroup of $G$ containing it. Glauberman's $Z^*$-theorem \cite{Glauberman} states that if $x\in G$ is an isolated involution in $G$, then $G=C_G(x)O_{2'}(G).$ The proof of this theorem depends on the modular representation theory and is independent of the classification of finite simple groups. Recall that for a prime $p$, $O_{p'}(G)$ is the largest normal $p'$-subgroup of $G$. Extending this fundamental theorem to all primes,  Glauberman's $Z_p^*$-theorem states that if $x\in G$ is an isolated $p$-element, then $G=C_G(x)O_{p'}(G)$. For various proofs of this theorem, see \cite{Ar, GGLN, GR, Tong-Viet}. Note that all of these proofs depend on the classification of finite simple groups. For  many equivalent statements of this theorem, see \cite{GTT}. Also, see \cite{FR} for some variant of Glauberman's $Z_p^*$-theorem.

In this paper, we introduce the so-called \emph{extremely closed subgroup} and obtain some new factorization of finite groups similar to Glauberman's $Z^*_p$-theorem which  gives some non-simplicity criteria for finite groups.
Let $G$ be a finite group and let $H\leq M$ be subgroups of $G$. We say that $H$ is  extremely closed in $M$ with respect to $G$ if $\la H,H^g\ra\cap M=H$ for all $g\in G-M,$ and that $H$ is extremely closed in $G$ if $H$ is extremely closed in $N_G(H)$ with respect to $G$. Trivially,  if $H$ is a normal or self-normalizing subgroup of $G,$ then $H$ is extremely closed in $G.$

 If $G$ is a finite group, we write $Z(G)$ for the center of $G$ and $\Phi(G)$ for the Frattini subgroup of $G$, that is,  the intersection of all maximal subgroups of $G$. Furthermore, if $H$ is a subgroup of $G$, then $\la H^G\ra$ is the normal closure of $H$ in $G$. We next compare our definition of extremely closed subgroups with other known embedding properties of subgroups of finite groups. 
 
The first motivation for our definition comes from work of Flavell  \cite{Fla1} on the generation of finite groups with maximal subgroups of maximal subgroups. In particular, a triple $(G,M,H)$ with $H\leq M\leq G$ is called a $\gamma$-triple if $H<M<G$ and $\la H,g\ra\cap M=H$ for all $g\in G-M.$ If $H$ is maximal in $M$ and $M$ is maximal in $G$ and moreover, $G$ cannot be generated by any two conjugates of $H$, then $(G,M,H)$ is a $\gamma$-triple. Clearly, if  $(G,M,H)$ is a $\gamma$-triple, then $H$ is extremely closed in $M$ with respect to $G$. The converse is not true by the example below.
  \begin{ex}\label{ex1}
  Let $G=\Sym_4$ be the symmetric group of degree $4$. Let $H=\la (1,2,3)\ra$ and $M=N_G(H)\cong \Sym_3$. Then $\la H^G\ra=\Alt_4$ and $\la H^G\ra\cap M=H.$ So $\la H,H^g\ra\cap M=H$ for all $g\in G$, and hence $H$ is extremely closed in $M$ with respect to $G$. However, let $g=(1,3,2,4)\in G-M$. Then $\la H,g\ra=G$ and so $\la H,g\ra\cap M=M\neq H$. Therefore $(G,M,H)$ is not a $\gamma$-triple.
  \end{ex}
 
 For the second motivation, following Hawkes and Humphreys, a subgroup $M$ of a finite group $G$ is said to have property CR (character restriction)  if every irreducible complex character of $M$ is the restriction of a character of $G$. In \cite{HH}, the authors studied finite solvable groups with a CR-subgroup  and the general cases were considered  by Isaacs in \cite{Isaacs}. One important property of a CR-subgroup $M$ of a finite group $G$ is that if $H\unlhd M$, then $\la H^G\ra\cap M=H$ (See \cite[Proposition 1.1]{Isaacs}).  Berkovich \cite{Berkovich} called  a triple $(G,M,H)$ with $H\unlhd M\leq G$  \textit{special} in $G$ if $\langle H^G\rangle \cap M=H.$ (In \cite{Li}, Li calls $H$ an NE-subgroup of $G$ if $(G,N_G(H),H)$ is special in $G$.) In \cite{Isaacs}, Isaacs showed that if $P$ is a Sylow $p$-subgroup of $G$, where $p$ is a prime and assume that $N_G(P)$ satisfies CR in $G$, then $N_G(P)$ has a normal complement in $G$. This result was extended by Berkovich \cite{Berkovich}, where he showed that if  both triples $(G,N_G(P),P)$  and $(G,N_G(P),\Phi(P))$ are special in $G$, then $N_G(P)$ has a normal complement in $G$. This gives a character theory free proof of Isaacs' result mentioned earlier. Observe that if a triple $(G,N_G(H),H)$ is special in $G,$ then $H$ is extremely closed in $G$. However, the converse is not true. 
 
\begin{ex}\label{ex2}
Let $G=P:H$  be a semidirect product of $H$ and $P,$ where $H=\la a\ra$ is a cyclic group of order $2$ and $P\cong 3_+^{1+2}$ is an extraspecial group of order $27$ with exponent $3,$  so $$P=\la x,y,z| z=[x,y],x^3=y^3=z^3=1=[x,z]=[y,z]\ra$$ and $H$ acts on $P$ via $x^a=x^{-1},y^a=y^{-1}$ and $z^a=z.$  Then $N_G(H)=C_G(H)=H\la z\ra,$ and $\la H^G\ra=G.$ For every $g\in G-N_G(H),$ we have $T=\la H,H^g\ra$ is a dihedral group of order $6,$ so  $N_T(H)=H,$ and hence $H$ is extremely closed in $G$ but $(G,N_G(H),H)$ is not a special triple since $\la H^G\ra\cap N_G(H)=H\la z\ra\neq H.$ 

\end{ex} 
  
Finally, we mention the last inspiration for our new embedding property.  Let $H\leq M$ be subgroups of a finite group $G$. Recall that $H$ is said to be {\em strongly closed} in $M$ with respect to $G$ if, whenever $a^g\in M,$ where $a\in H, g\in G$, then $a^g\in H.$ This is equivalent to saying that $M\cap H^g\leq H$ for all $g\in G.$ Furthermore, we say that $H$ is strongly closed in $G$ if $H$ is strongly closed in $N_G(H)$ with respect to $G$. Noting that in \cite{Bianchi}, $H$ is called an $\mathscr{H}$-subgroup of $G$ if $H$ is strongly closed in  $G$.
   If $H=\la x\ra$ is cyclic of order $2$, then $H$ is strongly closed in $G$ if and only if $x$ is isolated in $G$. Finite groups with a strongly closed $p$-subgroup are determined in  \cite{FF,Foote,Gold}.
It is easy to see that if $H$ is extremely closed in $G$, then $H$ is strongly closed in $G$.

 \begin{ex}\label{ex3}
Let $G=\UU_3(4)$. By \cite{Gold}, if $P$ is a Sylow $2$-subgroup of $G$, then $H=Z(P)=\Phi(P)$ is a strongly closed abelian $2$-subgroup of $G$. Using GAP \cite{GAP}, we can find $g\in G$ of order $15$ such that $\la H,H^g\ra\cong \Alt_5\cong \SU_2(4)$, $N_G(H)\cong P\la g\ra$ and $\la H,H^g\ra\cap N_G(H)\cong \Alt_4\neq H$. Thus  $H$ is not extremely closed in $G$.
\end{ex}

Finally, we recall the definition of weakly closed subgroups. Let $H\leq M$ be subgroups of a finite group $G$. We say that  $H$ is {\em weakly closed} in $M$ with respect to $G$ if, whenever $H^g\subseteq M,$ where $g\in G,$ then $H^g=H.$   It is  easy to see that if $H$ is strongly closed in $M$ with respect to $G$, then $H$ is weakly closed in $M$ with respect to $G$. Moreover, when $H$ is cyclic of prime order, these two concepts coincide.
We know that if $H\unlhd M\leq G$ and $\la H^G\ra\cap M=H$, then $H$ is extremely closed in $M$ with respect to $G$. In our first result, we show that in certain cases, the converse holds.

\begin{theorem}\label{th1} Let $G$ be a finite group and let $H\leq M$ be subgroups of $G$. Suppose that $\la H,H^g\ra\cap M=H$ for all $g\in G.$ If $H$ is maximal in $M$ and $M$ is maximal in $G,$ then  $\la H^G\ra\cap M=H.$
\end{theorem}

Let $H\leq M$ be subgroups of a finite group $G$. Recall that $H$ is called  a weak second maximal subgroup of $G$ if there exists a maximal subgroup $M$ of $G$ such that $H$ is maximal in $M$.  Moreover,  $H$ is called a second maximal subgroup of $G$ if $H\neq G$ and it is maximal in every maximal subgroup of $G$ containing it. In \cite{Fla1}, Flavell shows that if $G$ is a finite non-abelian simple group and $H$ is a weak second maximal subgroup of $G$, then $G=\la H,g\ra$ for some $g\in G.$ And in \cite{Fla2}, it is shown that if $H$ is a second maximal subgroup of a finite non-abelian simple group $G$, then $G=\la H,H^g\ra$ for some $g\in G.$ 
 As a corollary to Theorem \ref{th1}, we  obtain the following generation result for finite non-abelian simple groups. 

\begin{corol}\label{cor:generation}
Let $G$ be a finite non-abelian simple group. Let $M$ be a maximal subgroup of $G$ and let $H$ be a normal subgroup of $M$ of prime index. Then $G=\la H,H^g\ra$ for some $g\in G.$
\end{corol}

 Note that if $G=\Alt_5, M=\Sym_3$ and $H$ is a Sylow $2$-subgroup of  $M$, then $M$ is a maximal subgroup of $G$ and $|M:H|=3$ is a prime, but $G$ cannot be generated by any two conjugates of $H$. So we cannot drop the hypothesis that $H\unlhd M$ in Corollary \ref{cor:generation}.  We should mention that Flavell \cite{Fla3} asks whether a finite non-abelian simple group can be generated by two conjugates of a self-normalizing subgroup.

Let $G$ be a finite group and let $p$ be a prime. We now focus on extremely closed  $p$-subgroups. Let  $H$  be an extremely closed $p$-subgroup of $G$.   We first assume that $H=\la x\ra$ is cyclic of order $p.$ If $p=2,$ then it is not hard to see that $x$ is isolated in  $G$ or equivalently $H$ is strongly closed in $G$ and so $G=C_G(x)O_{2'}(G)$ by Glauberman's $Z^*$-Theorem. In particular, $G$ is not simple. When $p$ is odd, we note that there exists a simple group with a weakly closed or strongly closed subgroup of order $p,$ for instance, when a Sylow $p$-subgroup of $G$ is cyclic. However, when $H$ is extremely closed in $G$, the subgroups generated by any two distinct conjugates of $H$ are Frobenius groups. By applying a result due to B. Fischer \cite{Fis} concerning Frobenius automorphisms, we can show that $(G,N_G(H),H)$ is special in $G,$ and then we obtain a factorization similar to that of Glauberman's $Z_p^*$-Theorem.

\begin{theorem}\label{th2} Let $G$ be a finite group and let $p$ be an odd prime divisor of the order of $G$. Let $H$ be a cyclic subgroup of order  $p$ of $G$. If $\la H,H^g\ra\cap N_G(H)=H$ for all $g\in G$,  then $\la H^G\ra\cap N_G(H)=H$. In particular, $G=N_G(H)O_{p'}(G)$.
\end{theorem}

It would be nice to have a proof of Theorem \ref{th2} which does not rely on the classification of finite simple groups.
For an arbitrary prime $p$, if  $H$ is an abelian extremely closed $p$-subgroup of a finite group $G$,  we also obtain a similar factorization $G=N_G(H)O_{p'}(G)$ as in Theorem \ref{th2}.  For even prime, the proof depends only on the classification of finite groups with an abelian strongly closed subgroup by Goldschmidt \cite{Gold} which is independent of the classification of finite simple groups. For odd primes, we make use of a result due to Guest \cite{Guest} on the characterization of solvable radical of finite groups and the classification of finite groups with strongly closed subgroups by Flores and Foote \cite{FF}.

\begin{theorem}\label{th3} Let  $G$ be a finite group and let $p$ be prime. Let $H$ be an abelian $p$-subgroup of $G$. If $\la H,H^g\ra\cap N_G(H)=H$ for all $g\in G,$ then  $G=N_G(H)O_{p'}(G).$ 
\end{theorem}

 Remark that we cannot drop the hypothesis that $H$ is abelian when $p=2$ in Theorem \ref{th3}, since the simple group $G=\LL_2(17)$ has a self-normalizing Sylow $2$-subgroup $P$ which is non-abelian and so $P$ is clearly an extremely closed $2$-subgroup of $G$. Note that if $|H|=2$, then Theorem \ref{th3} is just Glauberman's $Z^*$- theorem. 
Example \ref{ex2}  above shows that  an abelian extremely closed $2$-subgroup $H$ may not satisfy the condition $\la H^G\ra\cap N_G(H)=H$. However, this holds true for odd primes. We obtain the following as a corollary to Theorem \ref{th3}.

 \begin{corol}\label{cor:special}
Let $G$ be a finite group and let $p$ be an odd prime. Let $H$ be an abelian $p$-subgroup of $G$. If $\la H,H^g\ra\cap N_G(H)=H$ for all $g\in G,$ then   $\la H^G\ra\cap N_G(H)=H$.
\end{corol}

Recall that for a finite group $G$, the solvable radical of $G$, denoted by $R(G)$, is the largest normal solvable subgroup of $G$.
Theorem \ref{th3} and Corollary \ref{cor:special} now yield the following.

\begin{corol}\label{cor}
Let $G$ be a finite group and let $p$ be a prime. If $H$ is an extremely closed abelian $p$-subgroup of $G$, then $H\subseteq R(G)$.
\end{corol}

By an application of Burnside's normal $p$-complement theorem and the solvability of finite groups admitting a fixed point free coprime group action, if $H$ satisfies the hypothesis of the corollary, then $\la H,H^g\ra$ is solvable for all $g\in G.$ Thus if $x\in H$, then $\la x,x^g\ra$ is solvable for all $g\in G.$ By the main results in \cite{GGKP,Guest}, if $p\ge 5$, then $x\in R(G)$ and hence $H\leq R(G)$. Thus the above corollary only provides new result when $p=2$ or $3$.

In general, if $x$ is a $p$-element and $\la x,x^g\ra$ is $p$-solvable for all $g\in G$, then it is not true that $x\in R_p(G)$, where $R_p(G)$ is the $p$-solvable radical of $G$, that is, $R_p(G)$ is the largest normal $p$-solvable subgroup of $G$. For a counterexample, consider $G=\UU_3(3)$ and $x\in G$  a transvection, so $x$ has order $3$ and the conjugacy class of $G$ containing $x$ has size $56$, then we can check that $\la x,x^g\ra$ is isomorphic to either $\la x\ra$ or $\SL_2(3)$ for every $g\in G.$ Clearly $\la x,x^g\ra$ is $3$-nilpotent and hence it is $3$-solvable for every $g\in G.$ There is also a counterexample when $p=2$ since if $x\in G$ is an involution, then $\la x,x^g\ra$ is $2$-nilpotent for every $g\in G.$ Recall that a finite group $G$ is $p$-nilpotent if it has a normal $p$-complement for some prime $p.$ On the other hand, it is proved in \cite{Flavell} that if $P$ is a Sylow $p$-subgroup of a finite group $G$ for some prime $p,$ then $G$ is $p$-solvable if and only if $\la P,g\ra$ is $p$-solvable for all $g\in G$. Generalizing this result, we can prove the following.

\begin{theorem}\label{th5}
Let $G$ be a finite group and let $p$ be a prime. Let $P$ be a Sylow $p$-subgroup of $G$. Then $G$ is $p$-solvable if and only if $\la P,P^g\ra$ is $p$-solvable for all $g\in G.$
\end{theorem}

 Our notation is standard. For finite group theory, we follow \cite{gor} and \cite{KS} and for  finite simple groups, we follow the notation in \cite{KL}.

For the organization of the paper, we collect some useful results in Section \ref{sec2}. We will prove Theorems \ref{th1}-\ref{th3} and the corollaries in Section \ref{sec3} and the last theorem will be proved in Section \ref{sec4}.

\section{Preliminaries}\label{sec2}

Let $G$ be a finite group. Recall that the Fitting subgroup of $G$, denoted by $F(G)$, is the largest nilpotent normal subgroup of $G$. The layer of $G$, denoted by $E(G)$, is the product of all components of $G$, where  a component of $G$ is a subnormal quasi-simple subgroup of $G$. A finite group $L$ is quasi-simple if $L$ is perfect and $L/Z(L)$ is a non-abelian simple group. The generalized Fitting subgroup of $G$, denoted by $F^*(G)$, is defined by $F^*(G)=F(G)E(G).$ As usual, if $H\leq G$, then $N_G(H)$ and $C_G(H)$ denote the normalizer and centralizer of $H$ in $G$, respectively. Finally, a finite group $G$ is almost simple with socle $S$ if there exists a finite non-abelian simple group $S$ such that $S\unlhd G\leq \Aut(S).$

Recall that a subgroup $H$ of $G$ is called {\em pronormal} (resp. {\em abnormal}) in $G$ if for any $g\in G,$ $H^g=H^u$ for some $u\in \langle H,H^g\rangle,$ (resp. $\langle H,H^g\rangle=\langle H,g\rangle$). 
The first lemma is obvious, for completeness, we will include  a proof here.
\begin{lem}\label{proab} Let $G$ be a finite group. Let $H$ be a pronormal subgroup of $G$ and let $N\unlhd G.$ Then the following hold.

$(i)$ If $N\unlhd G$ and $P\in \Syl_p(N)$, then $P$ is pronormal in $G$.

$(ii)$ $N_G(H)$ is abnormal in $G$.

$(iii)$ If $H\leq N\unlhd G$, then $G=N_G(H)N$.

$(iv)$ If $H\leq L\leq G$, then $H$ is pronormal in $L$.

$(v)$ If $H$ is subnormal in $K$, where $K\leq G,$ then $H\unlhd K$.

$(vi)$ If $L=\la H^G\ra$, then $\la H^L\ra=L.$
\end{lem}

\begin{proof} $(i)$ Let  $g\in G$. Since $P\leq N\unlhd G,$ we have $\la P, P^g\ra \leq N.$ As $P\in \Syl_p(N),$ it follows that  $P\in \Syl_p(\la P,P^g\ra)$ and hence by Sylow's Theorem, $P^g=P^u$ for some $u\in \la P,P^g\ra.$ Thus $P$ is pronormal in $G.$

$(ii)$ Assume that $H$ is pronormal in $G.$ Let $M=N_G(H)$ and  let $g\in G$. By definition,  $H^g=H^u,$ for some $u\in \la H,H^g\ra,$ whence $gu^{-1}\in N_G(H)=M.$ Since $u\in \la H,H^g\ra\leq \la M,M^g\ra,$ we have $g\in \la M,M^g\ra$ and so $\la M,M^g\ra=\la M,g\ra$.  Hence $M$ is abnormal in $G$.
 
$(iii)$ Let $g\in G.$ We have $H^g=H^u,$ where $u\in \la H,H^g\ra\leq N$ as $H\leq N\unlhd G.$ Thus $gu^{-1}\in N_G(H)$ and so $g\in N_G(H)N.$

$(iv)$ This is obvious.

$(v)$ By $(iv),$ it suffices to show that if $H$ is subnormal and pronormal in $G$ then $H\unlhd G.$ In fact, we only need to prove the following: if $H\unlhd K\unlhd G$ and $H$ is pronormal in $G$ then $H\unlhd G.$ By applying $(iii),$ we have $G=N_G(H)K.$ However, as $H\unlhd K,$ $K\leq N_G(H)$ and so $G=N_G(H).$

$(vi)$ Since $H\leq L=\la H^G\ra \unlhd G,$ $G=N_G(H)L$ by $(iii)$ and thus $$L=\la H^G\ra=\la H^{N_G(H)L}\ra=\la H^L\ra\leq L.$$ Therefore $L=\la H^L\ra.$ This completes the proof of the lemma.
\end{proof}

We next deduce some properties of extremely closed subgroups.
\begin{lem}\label{induct} Let $G$ be a finite group and let  $H$ be a $p$-subgroup of $G$ for some prime $p$.  Let $N\unlhd G$ and assume that $H$ is extremely closed in $G$. Let $\overline{G}=G/N.$ Then the following hold.

$(i)$ For every $g\in G,$ we have  $N_{\la H,H^g\ra}(H)=H$ and $H\in\Syl_p(\la H,H^g\ra)$.

$(ii)$ $H$ is pronormal in $G.$

$(iii)$ If $H\leq L$, then $H$ is extremely closed in $L$.

$(iv)$ $N_{\overline{G}}(\overline{H})=\overline{N_G(H)}.$

$(v)$ If $H$ is abelian, then  $\la H,H^g\ra=HO_{p'}(\la H,H^g\ra)$, for every $g\in G$.

$(vi)$ $\overline{H}$ is extremely closed in $\overline{G}$.

$(vii)$ If $H\leq Q\leq G,$ where $Q$ is a $p$-group, then $N_G(Q)\leq N_G(H).$ In particular, if $H\leq P\in \Syl_p(G)$, then $P\leq N_G(H).$

\end{lem}

\begin{proof}

 $(i)$ Let $g\in G$ and  let $T=\langle H,H^g\rangle.$ We have $H=T\cap N_G(H)=N_T(H)$ and so  $H$ is a Sylow $p$-subgroup of $T$ by Sylow's theorem.

 $(ii)$ Let $g\in G.$ As above, let $T=\la H,H^g\ra$. From part (i), $H$ is a Sylow $p$-subgroup of $T$ and since $|H^g|=|H|$ and $H^g\leq T,$ $H^g$ is also a Sylow $p$-subgroup of $T$. By Sylow's theorem, $H^g=H^u$ for some $u\in T.$ Thus $H$ is pronormal in $G.$

$(iii)$ Let $g\in L$. Then $\la H,H^g\ra\cap N_L(H)=\la H,H^g\ra\cap N_G(H)\cap L=H\cap L=H.$

$(iv)$ It suffices to show that $N_G(HN)\leq N_G(H)N.$ Let $g\in N_G(HN).$ Then $H^g\leq HN$ and hence $T=\la H,H^g\ra\leq HN.$ Since $H$ is pronormal in $G$ by $(ii),$ we have $H^g=H^u,$ for some $u\in T\leq HN.$ Thus $gu^{-1}\in N_G(H)$ whence $g\in N_G(H)HN=N_G(H)N.$

$(v)$ Assume that $H$ is abelian. Let $g\in G$ and let $T=\la H,H^g\ra.$ By (i), $H$ is a self-normalizing abelian Sylow $p$-subgroup of $T$. The result now follows from Burnside's normal $p$-complement theorem (\cite[Theorem $7.4.3$]{gor}).

$(vi)$ Applying $(iv),$ we need to show that $\la H,H^g\ra N\cap N_G(H)N=HN$, for all $g\in G.$ Let $T=\la H,H^g\ra.$
By Dedekind's Modular law, we have $$T N\cap N_G(H)N=N(TN\cap N_G(H)).$$ Hence, it suffices to show that $TN\cap N_G(H)\leq HN.$
Let $y=xn\in TN\cap N_G(H),$ where $y\in N_G(H),x\in T$ and $n\in N.$ We have $H=H^y=H^{xn},$ it implies that $H^x=H^{n^{-1}}.$ By (i), we have $N_T(H)=H,$ and so by $(ii)$ and Lemma \ref{proab}$(ii),$ $H$ is abnormal in $T.$ Thus $x\in \la H,H^x\ra= \la H,H^{n^{-1}}\ra\leq HN.$ Therefore, $y=xn\in HN.$

$(vii)$ Since $H$ is pronormal and subnormal in $N_G(Q),$ by Lemma \ref{proab}$(v),$ $H$ is normal in $N_G(Q).$ The remaining claim is obvious.
The proof  is now complete.
\end{proof}

\begin{lem}\label{strong} Let $G$ be a finite group, let $N\unlhd G$ and let $H$ be an extremely closed  $p$-subgroup of $G$ for some prime $p$.  Let $P$ be a Sylow $p$-subgroup of $G$ containing $H$ and let $Q=H\cap N$. Then the following hold.

$(i)$ $H$ is strongly closed in $P$ with respect to $G$.

$(ii)$  $Q$ is strongly closed in $P\cap N$ with respect to $N$.

$(iii)$ If  $N\leq N_G(H),$ then $Q\unlhd G.$

\end{lem}
\begin{proof} 
 Observe that $R:=P\cap N\in \Syl_p(N)$ and by Lemma \ref{induct}(vii), $P\leq N_G(H)$.
 
$(i)$ For $g\in G$, we have $H^g\cap P\leq \la H,H^g\ra \cap N_G(H)=H.$ So $H$ is strongly closed in $P$ with respect to $G$.

$(ii)$ For $n\in N,$ we have $Q^n\leq H^n$ and $R\leq P\leq N_G(H)$ and so $$Q^n\cap R\leq \la H,H^n\ra\cap N_G(H)= H.$$ Furthermore, as $Q\leq N\unlhd G$ and $R\leq N,$ we have  $Q^n\cap R\leq N.$ Hence we obtain $Q^n\cap R\leq H\cap N=Q.$

$(iii)$ Assume that $N\leq N_G(H).$ For each $g\in G,$ we have $$Q^g=Q^g\cap N\leq \la H,H^g\ra\cap N_G(H)\cap N=H\cap N=Q.$$ Hence $Q\unlhd G$ as wanted. The proof is now complete.
\end{proof}

We next quote some results that we will need for the proofs of the main theorems.
\begin{lem}\label{Guest1} Suppose that $G$ is a finite group with $F(G)=1.$ Let $L$ be a component of $G.$ If $x\in G$ such that $x\not\in N_G(L)$ and $x^2\not\in C_G(L)$ then there exists an element $g\in G$ such that $\la x,x^g\ra$ is not solvable.
\end{lem}

\begin{proof}
This is Lemma $1$ in \cite{Guest}.
\end{proof}

\begin{lem}\label{Guest2} Let $G$ be a finite almost simple group with socle $L.$ Suppose that $x\in G$ is an element of order $p,$ where $p$ is an odd prime. Then one of the following holds:

$(i)$ $\la x,x^g\ra$ is not solvable for some $g\in G;$

$(ii)$ $p=3$ and $L$ is a finite simple group of Lie type defined over $\FF_3,$ a finite field with $3$ elements, or $L\cong \UU_n(2), n\geq 4.$ Moreover, the Lie rank of $L$ is at least $2$ unless $L\cong \UU_3(3).$
\end{lem}

\begin{proof}
This is Theorem $A^*$ in \cite{Guest}.
\end{proof}

Recall that a triple $(G,M,H)$ with $H\unlhd M\leq G$ is called a $W$-triple if $M\cap M^g\leq H$ for all $g\in G-M.$

\begin{lem}\label{wiefla} Let $G$ be a finite group and $H\unlhd M\leq G.$ Then $(G,M,H)$ is a $W$-triple if and only if $$N_G(D)\leq M\mbox{ for all subgroups $D\leq M$ with $D\not\leq H.$}$$
\end{lem}

\begin{proof}
This is Lemma 2.3 in \cite{Fla1}.
\end{proof}

\begin{lem}\label{Wielandt}\emph{(Wielandt's Theorem).} Let $G$ be a finite group. If $(G,M,H)$ is a $W$-triple, then $G$ contains a normal subgroup $K$ such that $G=MK$ and $M\cap K=H.$ In particular, the triple $(G,M,H)$ is special in $G.$
\end{lem}

\begin{proof}
The first claim is in \cite{Wielandt} or \cite[Exercise 1, p. 347]{Suzuki} 
and the second is in \cite[Lemma 9]{Berkovich}.
\end{proof}
An automorphism $\theta$ of a finite group is Frobenius if each nontrivial power of $\theta$ is fixed point free.

\begin{lem}\label{Fisch}  Let $G$ be a finite group and let $D$ be a conjugacy class of $G$ containing elements of order $>2$. Assume that $G=\la D\ra$. Then some element of $D$ induces a Frobenius automorphism on $G'$ if and only if each pair of distinct elements in $D$ generates a Frobenius group.
\end{lem}

\begin{proof}
This is Satz I in \cite{Fis}.
\end{proof}

We also need the following results.

\begin{lem}\label{lem:generators}
Let $G$ be a finite group.

\begin{itemize}
\item[$(i)$] Let $\pi$ be a set of odd primes and suppose that the $\pi$-group $P$ acts as a group of automorphisms on the solvable finite $\pi'$-group $G$. Then \[C_{[G,P]}(P)=\la C_{[g,P]}(P):g\in G\ra.\]
\item[$(ii)$] Let $\alpha$ be a coprime automorphism of odd order of  $G$. Then \[C_{[G,\la \alpha\ra]}(\alpha)=\la C_{[g,\la \alpha\ra]}(\alpha):g\in [G,\la \alpha\ra]\ra.\]
\end{itemize}
\end{lem}

\begin{proof}
The first claim is \cite[Theorem A]{Fla5} and the second is \cite[Theorem 2]{FR}.
\end{proof}
We will use the next result repeatedly.
\begin{lem}\label{lem:coprime}
Let $\pi$ be a non-empty set of primes. Let  $Q$ be a finite $\pi$-group which acts fixed point freely on a finite $\pi'$-group $R$, that is, $C_R(Q)=1$, then $R$ is solvable.
\end{lem}

\begin{proof}
For a proof, see Theorem 2.3 in \cite{GMN}.
\end{proof}

We also need the following consequence of  Burnside's normal $p$-complement theorem.
\begin{lem}\label{lem:complement}
Let $G$ be a finite group and let $p$ be a prime. Let $H$ be an abelian $p$-subgroup of $G$. Assume that $\la H^G\ra\cap N_G(H)=H$. Then  $G=N_G(H)O_{p '}(G)$ and $\la H^G\ra$ is solvable.  Moreover,  $N_G(H)$ has a normal complement in $G$, which is $O_{p'}(\la H^G\ra)$.
\end{lem}

\begin{proof}
Let $H$ be an abelian $p$-subgroup of $G$. Assume that $\la H^G\ra\cap N_G(H)=H$. Let $L=\la H^G\ra\unlhd G.$ Then $N_L(H)=L\cap N_G(H)=H$, so $H$ is a self-normalizing abelian Sylow $p$-subgroup of $L$. In particular, $H\leq Z(N_L(H))$ and hence by Burnside's normal $p$-complement theorem (\cite[Theorem 7.4.3]{gor}) $L=HO_{p'}(L)$. By Frattini's argument $$G=N_G(H)L=N_G(H)O_{p'}(L)=N_G(H)O_{p'}(G).$$ The last equality holds since $O_{p'}(L)\leq O_{p'}(G)$.  

Let $L=\la H^G\ra$. Then $N_L(H)=H$, that is, $H$ is a self-normalizing cyclic Sylow $p$-subgroup of $L$.  By Burnside's normal $p$-complement theorem,  $L$ has a normal $p$-complement $K$ , and hence $C_K(H)=1$. By Lemma  \ref{lem:coprime}, $K$ is solvable and thus $L=HK$ is solvable as well.

We now show that  $O_{p'}(L)$ is a normal complement to $N_G(H)$ in $G$. To see this, observe that $L=HO_{p'}(L)$ and  $H\cap O_{p'}(L)=1.$ Note that $O_{p'}(L)\unlhd G$.  
Thus it suffices to show that $N_G(H)\cap O_{p'}(L)=1.$ Indeed, we have \[N_G(H)\cap O_{p'}(L)=N_G(H)\cap L\cap O_{p'}(L)=H\cap O_{p'}(L)=1.\]
The proof is now complete.
\end{proof}

Finally, we also need the following solvability result.
\begin{lem}\label{lem:solvable}
Let $G$ be a finite group and let  $H$ be an abelian $p$-subgroup of $G$ for some prime $p$. If $H$ is extremely closed in $G$, then $\la H,H^g\ra$ is solvable for all $g\in G.$
\end{lem}

\begin{proof}

Let $g\in G$ and  let $T=\la H,H^g\ra.$ By Lemma \ref{induct}$(v),$ we have $T=HO_{p'}(T).$ Since $N_T(H)=H,$ $H$ acts fixed point freely and coprimely on $O_{p'}(T),$ the claim now follows from Lemma \ref{lem:coprime}.
\end{proof}
\section{Extremely closed abelian $p$-subgroups}\label{sec3}

We are now ready to prove the main theorems of the paper. We first prove Theorem \ref{th1}.

\smallskip

\begin{proof}[{\bf Proof of Theorem \ref{th1}.}] 
Assume that $M$ is a maximal subgroup of $G$, $H$ is a maximal subgroup of $M$  and that $\la H,H^g\ra\cap M=H$ for all $g\in G$. We will show that $\la H^G\ra\cap M=H.$ The hypothesis implies that $H\unlhd M$ and $|M:H|$ is a prime.

Clearly, if $H\unlhd G$, then the conclusion holds. So,  assume that $M=N_G(H).$ 
Suppose that there exists a subgroup $D\leq M$ with $D\nleq H$ and $N_G(D)\nleq M$. Take $g\in N_G(D)- M$. Then $M=DH$
 and $M^g=DH^g$. As $M\neq M^g$, $G=D\la H,H^g\ra.$
 We have $$\la H^G\ra=\la H^{D\la H,H^g\ra}\ra\leq \la H,H^g\ra\leq \la H^G\ra$$ and so $\la H^G\ra=\la H,H^g\ra.$ Thus $\la H^G\ra\cap M=\la H,H^g\ra\cap M=H$ and we are done. 

Therefore, we can assume that whenever $D\leq M$ with $D\not\leq H$ then $N_G(D)\leq M.$ By Lemma \ref{wiefla}, $(G,M,H)$ is a $W$-triple and so the result follows from Lemma \ref{Wielandt}.
\end{proof}

\begin{proof}[\bf{Proof of Corollary \ref{cor:generation}}]
Let $G$ be a finite non-abelian simple group. Let $M$ be a maximal subgroup of $G$ and let $H$ be a normal subgroup of $M$ such that $|M:H|=p$ is a prime. Suppose by contradiction that $G\neq \la H,H^g\ra$ for all $g\in G.$  Since $G$ is non-abelian simple, either $\la H^G\ra=G$  or $H=1$. If $H=1$, then $M$ is a self-normalizing cyclic subgroup of $G$ of prime order $p$ and Burnside's normal $p$-complement theorem implies that $G$ has a normal $p$-complement, a contradiction.  

So $H\neq 1$ and $\la H^G\ra=G$. In particular, $M=N_G(H)$ and $\la H^G\ra\cap M=M>H$. Since $H$ is maximal in $M$ and $M$ is maximal in $G$, Theorem \ref{th1} implies that there exists $g\in G$ such that $\la H,H^g\ra\cap M>H$.
 The maximality of $H$ in $M$ implies that $M\leq \la H,H^g\ra$.  Hence $M=\la H,H^g\ra$ and $N_G(H^g)=M^g\neq M$ as $g\not\in M.$ 
Let $a\in M^g-M$, then $H^aH^g=M^a$ is a subgroup and $H^g\la H,H^a\ra=G$. Now $G=\la H^G\ra=H^{\la H, H^a\ra}\leq \la H,H^a\ra$, contradicting our assumption. Thus $G=\la H,H^g\ra$ for some $g\in G$ and the proof is now complete.
\end{proof}
\smallskip

\begin{proof}[{\bf Proof of Theorem \ref{th2}.}] 

Suppose we have proven that $\la H^G\ra\cap N_G(H)=H$.
By Lemma \ref{lem:complement}, we have $G=N_G(H)O_{p'}(G)$ and $\la H^G\ra$ is solvable. 

It remains to show that if $H$ is extremely closed in $G$, then $\la H^G\ra\cap N_G(H)=H$.
Let $G$ be a counterexample to the claim with  minimal order.   Then $H$ is a cyclic group of odd prime order $p$, $\la H,H^g\ra\cap N_G(H)=H$ for all $g\in G$ but $\la H^G\ra\cap N_G(H)\neq H.$ Furthermore, since $|H|=p$ is prime, $H^g=H$ or $\la H,H^g\ra$ is a Frobenius group for all $g\in G$.

\smallskip
(1) We first claim that $\la H^G\ra=G.$  Suppose by contradiction that $L:=\la H^G\ra<G.$ By Lemma \ref{induct}$(ii),$ $H$ is pronormal in $G$ and hence $L=\la H^L\ra$ by Lemma \ref{proab}$(vi).$  By Lemma \ref{induct}$(iii),$ $H$ is extremely closed in $L$ and so by the minimality of $G,$  $\la H^L\ra \cap N_L(H)=H.$ However, as $L=\la H^L\ra,$ we have $N_L(H)=H,$ and thus $$\la H^G\ra \cap N_G(H)=L\cap N_G(H)=N_L(H)=H.$$  This contradiction proves the claim.

\smallskip
(2) Assume  that $N_G(H)=C_G(H)$. Write $H=\la x\ra$. Then  $\la x,x^g\ra$ is a Frobenius group for all $g\in G.$ Hence $x$ acts as a Frobenius automorphism on $G'$ by Lemma \ref{Fisch}  and so $G'H$ is a Frobenius group (as $H$ has prime order).  In particular, $G'\leq O_{p'}(G)$ and $G=N_G(H)O_{p'}(G)$ in this case. In addition, $\la H^G\ra=G'H$ and $\la H^G\ra\cap N_G(H)=H,$ a contradiction.
 
 \smallskip
 (3) Assume that $N_G(H)>C_G(H)$. Let $M$ be a maximal normal subgroup of $G$. Then $G/M$ is a simple group. 
 Since $G=\la H^G\ra$, we have $H\nleq M.$  Assume that $M>1$. Then $|G/M|<|G|$. Lemma \ref{induct}(vi) implies that $HM/M$ is extremely closed in $G/M$. Hence, as $G/M$ is simple, $G=HM$. Now $N_G(H)=N_M(H)H=C_G(H)$, which is a contradiction. Hence $M=1$ and $G$ is a non-abelian simple group.

  \smallskip
 (4) By Lemma \ref{lem:solvable}, $\la x,x^g\ra$ is solvable for all $g\in G$ and so  by Lemma \ref{Guest2}, $p=3$ and $G$ is a finite simple group of Lie type defined over $\FF_3$ or $G\cong \UU_n(2)$ with $n\ge 4.$ Furthermore, except for $\UU_3(3)$, the Lie rank of $G$ is strictly greater than $1$. Now it is easy to see that a Sylow $3$-subgroup of $G$ is non-abelian.  Let $P$ be a Sylow $3$-subgroup of $G$ containing $x$. Then $H=\la x\ra$ is isolated in $P$ with respect to $G$, that is, $H$ does not conjugate in $G$ to any subgroup in $P-H$. By \cite[Theorem 4.250]{Gorenstein}, we deduce that $G\cong\UU_3(3)$.  By \cite[Theorem A*]{Guest}, $x$ is a transvection.  Now for any conjugate $x^g$ of $x$ different from $x$ and $x^{-1}$, we have that $T:=\la x,x^g\ra\cong\SL_2(3)$. However, $\SL_2(3)$ is not a Frobenius group. 
 The proof is now complete.
  \end{proof}
 
 \bigskip

We will need the following result which is a consequence of Theorem $A^*$ in \cite{Guest} and Theorem 1.2 in \cite{FF}.
\begin{prop}\label{prop:Guest}
Let $G$ be a finite group and let $p$ be an odd prime. Let $H$ be a nontrivial abelian $p$-subgroup of $G$. Assume that $G$ has a unique minimal normal subgroup $N$ which is nonabelian such that $G=HN$. Then $H$ is not extremely closed in $ G.$
\end{prop}

\begin{proof}
Suppose by contradiction that $H$ is an abelian extremely closed $p$-subgroup of $G$. By Lemma \ref{lem:solvable}, we have $\la H,H^g\ra$ is solvable for all $g\in G.$ In particular, if $x\in H$, then $\la x,x^g\ra$ is solvable for all $g\in G.$

By the uniqueness of $N$, we have $G=\la H^G\ra=HN$. We first show that $G$ is an  almost simple group with socle $S$.
Let $W\unlhd N$ be a component of $G.$ Assume that $W\neq N.$ As $W\unlhd N,$ $N\leq N_G(W)$ and so $N_G(W)=N_H(W)N.$ Since $W<N,$ $N_G(W)<G$ and so $N_H(W)<H.$ Let $x\in H-N_H(W).$ Then $x\not\in N_G(W)$ and $x^2\not\in C_G(W)$ since $p$ is odd, whence $x$ and $W$ satisfy the hypothesis of Lemma \ref{Guest1} and hence $\la x,x^g\ra$ is not solvable for some $g\in G$, a contradiction. Thus $W\unlhd G$ and so $G=HN,$ where $N$ is a non-abelian simple group which is also a minimal normal subgroup of $G.$ Thus $G$ is an almost simple group with socle $S$ as wanted.

Let $x\in H$ be an element of order $p$. Then $\la x,x^g\ra$ is solvable for all $g\in G.$ By Lemma \ref{Guest2},  $p=3$ and $S$ is a finite simple group of Lie type defined over $\FF_3$ or $S\cong \UU_n(2), n\geq 4.$ If $|H|=3,$ then $\la H^G\ra$ is solvable by Theorem \ref{th2} and Lemma \ref{lem:complement}, which is impossible. Thus we may assume that $|H|\geq 9.$ Since $G=HS$, we have $G/S\cong H/(H\cap S)$ is a $3$-group and thus if $G\neq S$, then elements in $H-S$ induce outer automorphisms of $S$ of $3$-power order. 
 
\smallskip
$(i)$ Assume $S\cong \UU_n(2),n\geq 4$ or $G\neq S.$  If $S\cong \UU_n(2),$ then since  $n\geq 4, $ we have $|\Out(\UU_n(2))|=2(n,3)$. If $G\neq S$ and $S\not\cong \UU_n(2)$, then $S\cong {\rm D}_4(3)$ or ${}^3{\rm D}_4(3)$ since $S$ has no nontrivial field automorphism so $\Out(S)$ contains diagonal automorphisms and possibly graph automorphisms only. In all cases, the Sylow $3$-subgroup of $\Out(S)$ has order at most  $3$. Hence $|G:S|=|H: S\cap H|\leq 3$. 

Since $|H|\ge 9$, if $G$ is not simple then  $A=H\cap S> 1.$ By Lemma \ref{strong}(ii), $A$ is
strongly closed in $S.$ If $G=S=\UU_n(2)$, then $A=H$ is strongly closed in $S$. In both cases, $S$ contains a nontrivial abelian $3$-subgroup $A$ which is
strongly closed in $S.$ Since $S$ is simple, $S=\la A^S\ra.$   As a Sylow $3$-subgroup of $S$ is non-abelian,  $A$ cannot be a Sylow
$3$-subgroup of $S$; however  this contradicts  \cite[Theorem 1.2 (i)]{FF}.

\smallskip
$(ii)$ $G=S$ is a finite simple group of Lie type defined over $\FF_3$. The possibilities for $G$ can be read off from Table 1 in \cite{Guest} and note the correction in Remark 5.2 in \cite{PS}.

 Let $P$ be a Sylow $p$-subgroup of $G$ containing $H,$ and let $M$ be a maximal parabolic subgroup of $G$ containing the Borel subgroup $B:=N_G(P)$.
Let $R=O_p(M).$ Then $M=N_G(R)$. By induction, we
have $\la H^M\ra\cap N_M(H)=H,$ and hence $\la H^M\ra=HU,$ where $U=O_{p'}(\la H^M\ra)$ and
$C_U(H)=1.$ Therefore, $U$ is a solvable normal $p'$-subgroup of $M$ since $\la H^M\ra\unlhd M$. As $H\leq \la H^M\ra \unlhd M$ and $H$ is pronormal in $M$, we have $M=N_M(H)U$. However, as $M$ is a parabolic subgroup
of a finite simple group of Lie type $G$,  $F^*(M)=O_p(M)$ (Corollaries 3.1.4 and 3.1.5 in \cite{GLS}), so  $U=1.$ Hence $H\unlhd M$ and $M=N_G(H)$.
Thus $M$ contains every maximal parabolic subgroup of $G$ that contains the Borel subgroup
$B.$ However this happens only when the Lie rank of $G$ is $1.$ Therefore $G\cong \UU_3(3).$ (We can also use Theorem 1.2 in \cite{FF} to arrive at this conclusion.)

From the Atlas \cite{atlas}, we have $P\cong 3_+^{1+2}$ is an extraspecial group of order $27$ and exponent $3$. Thus
$|H|=9$. Since $G$ has Lie rank $1$, the Borel subgroup $B$ of $G$ is a maximal subgroup of $G$. Hence $B=N_G(H).$ Let $g\in G-B$ and $T=\la H,H^g\ra$. We know that $T=HO_{3'}(T)$ is solvable, so $T<G$ and thus $T$ lies in some maximal subgroup of $G$ whose order must be divisible by $|H|=9.$ Inspecting the list of maximal subgroups of $G$ in the Atlas \cite{atlas}, the only maximal overgroups of $H$ and $T$ in $G$ are the Borel subgroups $B$ and its $G$-conjugates.  Hence  $\la H,H^g\ra \leq B^t$ for some $t\in G.$ Note that $B=P:W$, where $W=\la y\ra$ is a cyclic group of order $8$. Since $B^t$ has a normal Sylow $3$-subgroup, we must have that $H\leq P^t\unlhd B^t$ and hence $H$ is subnormal in $B^t.$ Since $H$ is pronormal in $G$ and hence in $B^t,$ we have $H\unlhd B^t$ or $B^t=N_G(H)=B$. Therefore, $\la H,H^g\ra\leq B$ for all $g\in G.$ However, this implies that $\la H^G\ra\leq B$, which is a contradiction.
The proof is now complete.
\end{proof}

\begin{proof}[\bf{Proof of Theorem \ref{th3}}.]
Let $G$ be a counterexample to Theorem \ref{th3} with  minimal order. Then $\la H,H^g\ra\cap N_G(H)=H$ for all $g\in G$ but $G\neq N_G(H)O_{p'}(G)$, where $H$ is an abelian $p$-subgroup of $G$.

\smallskip
(1) $\la H^G\ra=G.$  
Let $L=\la H^G\ra.$ Assume  $L\neq G.$ By Lemmas \ref{induct}$(ii)$ and \ref{proab}$(iii),$ we have $G=N_G(H)L.$ By Lemma \ref{induct}$(iii)$ and the minimality of $G,$ we have $L=N_L(H)O_{p'}(L).$ As $L\unlhd G,$ $O_{p'}(L)\leq O_{p'}(G),$ and hence $$G=N_G(H)L=N_G(H)O_{p'}(L)=N_G(H)O_{p'}(G).$$ This contradictions shows that $L=G.$

\smallskip
(2) $O_{p'}(G)=1.$ Suppose by contradiction that $N=O_{p'}(G)\neq 1.$ It follows from Lemma \ref{induct}$(vi)$ that the hypothesis carries over to $\overline{G}=G/N,$   and so by the minimality of $G,$ we obtain $\overline{G}=N_{\overline{G}}(\overline{H})O_{p'}(\overline{G}).$ But $O_{p'}(\overline{G})=O_{p'}(G/O_{p'}(G))=1,$ and hence $\overline{G}=\overline{N_G(H)}$ by Lemma \ref{induct}$(iv).$ Therefore, $G=N_G(H)N=N_G(H)O_{p'}(G),$ which is a contradiction. This proves the claim.

\smallskip
(3) $O_p(G)=1.$ Assume $O_p(G)\neq 1.$ Let $N\leq O_p(G)$ be a minimal normal subgroup of $G$ and let $U=H\cap N.$ Observe first that $O_p(G)\leq N_G(H)$ since $P\leq N_G(H)$ for any $H\leq P\in \Syl_p(G).$ It follows that $N\leq N_G(H).$ We will show that $N\leq Z(G).$ Suppose first that $U\neq 1.$ Lemma \ref{strong}(iii) implies that $U\unlhd G.$ By the minimality of $N,$ we have $U=N,$ and so $N\leq H.$ As $H$ is abelian, $H\leq C_G(N)\unlhd G.$ By (1), we have  $C_G(N)=G,$ and $N\leq Z(G).$  If $U=H\cap N=1,$ then as $H$ and $N$ are both normal in $N_G(H),$ we have $[H,N]\leq H\cap N=1$ and thus $H\leq C_G(N)\unlhd G,$ so  $N\leq Z(G).$

By Lemma \ref{induct}$(vi),$ the hypothesis carries over to $\overline{G}=G/N,$   and so by the minimality of $G,$ we obtain $\overline{G}=N_{\overline{G}}(\overline{H})O_{p'}(\overline{G}).$ Let $K\leq G$ be such that $N\leq K$ and $K/N=O_{p'}(G/N).$ Then $N\unlhd K\unlhd G$ and  $|K:N|$ is odd.  By the Schur-Zassenhaus Theorem, $K=NT$ where $T$ is  Hall $p'$-subgroup of $G$. Moreover, since $N$ is central in $G$, $T\unlhd K$ and hence $T\unlhd G$ as it is characteristic in $K\unlhd G$. It follows that $T\leq O_{p'}(G)=1$ and thus $G=N_G(H)K=N_G(H)N=N_G(H).$
This contradiction proves the claim.

\smallskip
(4) Let $N$ be a minimal normal subgroup of $G$.  We claim that $G=HN$ and $N$ is the unique minimal normal subgroup of $G$. By (2) and (3),  $N\cong S^k$ for some finite non-abelian simple group $S$ with $p\mid |S| $ and some integer $k\ge 1$. Let $M=HN.$ Suppose that $M<G.$ By the minimality of $G$, we have $M=N_M(H)O_{p'}(M).$ Hence $O_{p'}(M)\leq N,$ and so $O_{p'}(M)=1$. We deduce that $H\unlhd M.$ As $H$ is an abelian $p$-group and $N\cong S^k$, we have $H\cap N=1$ and hence  $[H,N]\leq H\cap N=1 $. Thus $H\leq C_G(N)\unlhd G.$ By (1), we have $C_G(N)=G,$ and then $N\leq Z(G).$ This contradiction shows that $G=HN.$ Since $G/N\cong H/H\cap N$ is abelian, $N$ must be a unique minimal normal subgroup of $G$ as wanted.

\smallskip
If $p$ is odd,  then Proposition \ref{prop:Guest} yields a contradiction.
Thus for the remaining, we assume that $p=2$.

\smallskip
(5)  We next claim that $G$ is finite non-abelian simple group.   By (4), $F^*(G)=N$ and $H$ is a strongly closed abelian $2$-subgroup of $G$. Now by \cite[Theorem A]{Gold}, we have $F^*(G)=G$ and so $N=G.$ It follows that $G=S$ is simple. 
By \cite[Theorem A]{Gold} again, $G$ is isomorphic to one of the following groups:

$(i)$ $\LL_2(2^n),n\geq 3;$ ${}^2{\rm B}_2(2^{2n+1}),n\geq 1;$ or  $\UU_3(2^n),n\geq 2.$

$(ii)$ $\LL_2(q),q\equiv 3,5 \mbox{ (mod $8$)}.$

$(ii)$ ${}^2{\rm G}_2(3^{2n+1}),n\geq 1;$ or ${\rm J}_1,$ the first Janko group.

By Glauberman's $Z^*$-theorem, we may assume that $|H|\ge 4.$

(6) The final contradiction. We now consider each case above separately.

(a) Assume $G$ is isomorphic to one of the groups in $(i).$
Let $H\leq P\in \Syl_2(G)$ and let $B=PT$ be the Borel subgroup of $G$ containing $P.$ By \cite[(3.2)]{Gold}, $H= Z(P)$ is a non-cyclic elementary abelian $2$-group, $P$ is a $T.I$ subgroup of $G,$ and $B$ is the unique maximal subgroup of $G$ containing $P.$ It follows that $B=N_G(H)$. Observe that  for any $1\neq x\in H,$ $P\leq C_G(x)\leq B,$ as $P$ is uniquely contained in $B.$ For $g\in G-B,$ let $T=\la H,H^g\ra.$ Then $N_T(H)=H$ and then by Burnside's normal  $p$-complement theorem, we have $T=HU,$ where $U=O_{2'}(T)\unlhd T.$ As $H$ is non-cyclic abelian, $U=\la C_U(a):1\neq a\in H\ra$ (see \cite[8.3.4]{KS}). However, as $C_G(a)\leq B$ for all $1\neq a\in H,$ we have $U\leq B$ and so $T=\la H,H^g\ra\leq B.$ From the hypothesis, we must have $T=T\cap B=H,$ hence $H^g=H.$ This implies that $g\in B,$ contradicting the choice of $g.$

(b) Assume that $G\cong \LL_2(5)\cong \Alt_5.$ By \cite[(3.4)]{Gold}, $H\in \Syl_2(G),$ $ N_G(H)\cong \Alt_4$ and $N_G(H)$ is the unique maximal subgroup of $G$ containing $C_G(a),$ for all $1\neq a\in H.$ Take $g\in G-N_G(H),$ and let $T=\la H,H^g\ra.$ Then $T=HU,$ where $U=O_{2'}(T).$ As $H$ is non-cyclic, $U=\la C_U(a):1\neq a\in H\ra\leq N_G(H).$ This leads to a contradiction as in the previous case.

(c) Assume $G$ is isomorphic to one of the groups in $(ii)$ with $q\geq 11.$ By \cite[(3.4)]{Gold}, we have $H\in \Syl_2(G)$ and $N_G(H)\cong \Alt_4.$  Clearly, $G$ contains a maximal subgroup $M$ isomorphic to the dihedral group $D_{q\pm 1}$ such that $M$ does not contain $H$. Assume $M$ is generated by two involutions $a,b.$  We can choose $a,b$ such that $b\in H.$ Now $G=\la M,H\ra\leq \la a,H\ra,$ and hence $G=\la a,H\ra.$ By Sylow's Theorem, there exists some $g\in G$ such that $a\in H^g.$ Thus $G=\la a,H\ra\leq \la H^g,H\ra,$ and then $G=\la H,H^g\ra,$ which contradicts the hypothesis that $\la H,H^g\ra \cap N_G(H)=H.$

(d) Finally, assume that $G$ is isomorphic to one of the groups in $(iii).$  By \cite[(3.4)]{Gold}, we have $H\in \Syl_2(G).$ Now $H$ contains an involution $t$ such that $C_G(t)=\la t\ra\times L,$ where $L\cong \LL_2(q),$ $q\equiv 3,5$  {(mod $8$)} and $[G:C_G(t)]$ is odd (see \cite{Wal}). As $H\leq C_G(t)<G,$ by the minimality of $G$, we obtain $C_G(t)=(C_G(t)\cap N_G(H))O_{2'}(C_G(t)).$ However, as $C_G(t)=\la t\ra\times L,$ where $L$ is non-abelian simple, it follows that $O_{2'}(C_G(t))=1$ whence $C_G(t)\leq N_G(H)$. Hence $C_G(t)=N_G(H)$ since $C_G(t)$ is maximal in $G$ by \cite[$(3.4)$]{Gold}. It follows that $H\unlhd C_G(t)$ and then  $H\cap L\unlhd L,$ where $H\cap L\in \Syl_2(L),$ which contradicts  the simplicity of $L.$  
\end{proof}


\begin{proof}[\bf{Proof of Corollary \ref{cor:special}}]
Let $G$ be a counterexample to the corollary with minimal order. Then $H$ is an abelian $p$-subgroup for some odd prime $p$ and $\la H,H^g\ra\cap N_G(H)=H$ for all $g\in G$ but $\la H^G\ra\cap N_G(H)\neq H.$ By Theorem \ref{th3}, we have $G=N_G(H)O_{p'}(G)$. Now we have that  $$\la H^G\ra=\la H^{N_G(H)O_{p'}(G)}\ra=\la H^{O_{p'}(G)}\ra\leq HO_{p'}(G).$$
Let $L=\la H^G\ra.$ Then $L=H(L\cap O_{p'}(G))=HO_{p'}(L)$ and so $L$ has a normal $p$-complement. Assume that $L<G$. By the minimality of $G$, we have $\la H^L\ra\cap N_L(H)=H.$ However, $L=\la H^L\ra$ by Lemmas \ref{induct}(ii) and  \ref{proab}(vi) and thus $N_L(H)=H$ or equivalently $\la H^G\ra\cap N_G(H)=H$, a contradiction. Therefore, we can assume that $G=\la H^G\ra=HO_{p'}(G)$. 

Let $U=O_{p'}(G)$ and let $Q=[U,H]$. By \cite[8.2.7]{KS}, $U=QC_U(H)$ and $Q=[Q,H]$. Furthermore, $N_G(H)=C_G(H)=HC_U(H)$ and thus $G=HC_U(H)Q$. As $G=\la H^G\ra=H^{HC_U(H)Q}\leq HQ$, we obtain $G=HQ$ whence $U=Q$.

Let $N\leq U$ be a minimal normal subgroup of $G$. Since $HN/N$ is an abelian extremely closed $p$-subgroup of $G/N$ by Lemma \ref{induct}(vi) and  $G/N=\la (HN/N)^{G/N}\ra$, we have $N_G(HN)=HN$ and hence by Lemma \ref{induct} (iv), $N_G(H)N=HN.$ Moreover, $C_{U/N}(HN/N)=1$. It follows that $U/N$ is solvable by Lemma \ref{lem:coprime}.

Assume that $N$ is abelian. Then $U$ is solvable. Let $1\neq u\in U$ and let $T=\la H,H^u\ra.$ Then $T=H[u,H],$ where $[u,H]=O_{p'}(T)$ and $C_{[u,H]}(H)=1$ as $H$ is self-normalizing in $T.$ By Lemma \ref{lem:generators}(i), $C_{[U,H]}(H)=\la C_{[u,H]}(H):u\in U\ra=1.$ Thus $C_U(H)=C_{[U,H]}(H)=1$ and so $N_G(H)=C_G(H)=HC_U(H)=H.$ Therefore, $\la H^G\ra\cap N_G(H)=H$, which is a contradiction.

Assume that $N\cong S^k$, where $S$ is a nonabelian simple group and $k\ge 1$ is an integer. 
  Assume that $K=HN<G.$ By the minimality of $G,$ we have $\la H^K\ra\cap N_K(H)=H.$ Thus $\la H^K\ra=HQ$ and $C_Q(H)=1,$ where $Q=O_{p'}\la H^K\ra.$ As $Q$ is characteristic in $\la H^K\ra\unlhd K,$ it follows that $Q \unlhd K.$ Since $|K:N|$ is a power of $p$ and $Q$ is a $p'$-group, we must have $Q\leq N$ and hence $Q\unlhd N$. By \cite[1.7.5]{KS}, $Q$ is isomorphic to a direct product of the non-abelian simple group $S,$ so $Q$ is not solvable or $Q=1$. If the former case holds, then since $C_Q(H)=1$ and $Q$ is a $p'$-group,  Lemma \ref{lem:coprime} implies that $Q$ is solvable, which is a contradiction as it is a direct product of copies of $S$. Therefore, $Q=1$ and hence $H\unlhd K$. It follows that $R=H\cap N\unlhd N$. However, as $N\cong S^k$ and $H\cap N$ is a normal $p$-subgroup of $N$, we must have $H\cap N=1$. Hence $[H,N]\leq H\cap N=1$ and  so $H\leq C_G(N)\unlhd G$. Since $G=\la H^G\ra$, we have $C_G(N)=G$ or $N\leq Z(G)$, a contradiction. Therefore $G=HN$ and since $G/N$ is solvable, $N$ is a unique minimal normal subgroup of $G$. Now Proposition \ref{prop:Guest} yields a contradiction. The proof is now complete. 
\end{proof}

\begin{proof}[\bf{Proof of Corollary \ref{cor}}]
Let  $H$ be an extremely closed abelian $p$-subgroup of $G$ for some prime $p$. Assume first that $p=2$. By Theorem \ref{th3}, $G=N_G(H)O_{2'}(G)$ which implies that $HO_{2'}(G)\unlhd G$ and clearly $HO_{2'}(G)$ is solvable and thus $H\subseteq R(G).$ Assume now that $p$ is odd. By Corollary \ref{cor:special}, we have $N_G(H)\cap \la H^G\ra=H$  and thus $\la H^G\ra$ is solvable by Lemma \ref{lem:complement}. The proof is now complete.
\end{proof}

\section{A $p$-solvability criterion}\label{sec4}

Let $p$ be a prime. A finite group $G$ is said to be a minimal non-$p$-solvable group if $G$ is not $p$-solvable but every proper subgroup of $G$ is $p$-solvable.  A minimal simple group is a non-abelian finite simple groups whose all proper subgroups are solvable. Observe that  minimal non-$2$-solvable simple groups are exactly the minimal simple groups and these groups are classified by Thompson in \cite{Thompson}.

\begin{lem}\label{lem:minimal-simple}
Every minimal simple group is isomorphic to one of the following simple groups:
\begin{itemize}

\item[$(1)$] $\LL_2(2^r), r$ is a prime.

\item[$(2)$] $\LL_2(3^r), r$ is an odd prime.

\item[$(3)$] $\LL_2(r), r>3$ is a prime such that $5\mid r^2+1$.

\item[$(4)$] ${}^2{\rm B}_2(2^r),r$ is an odd prime.

\item[$(5)$] $\LL_3(3).$
\end{itemize}
\end{lem}
\begin{proof}
This is \cite[Corollary 1]{Thompson}.
\end{proof}
The next result classifies minimal non-$3$-solvable simple groups.
\begin{lem}\label{lem:minimal-3}
Let $G$ be a finite non-abelian simple group. Assume that every proper subgroup of $G$ is $3$-solvable. Then $G$ is isomorphic to a minimal simple group or to the Suzuki group ${}^2{\rm B}_2(q)$ with $q=2^{2m+1},m\ge 1.$
\end{lem}

\begin{proof}
This is Lemma 5.3 in \cite{GT}.
\end{proof}

Finally, we need the classification of finite non-$p$-solvable simple groups for any primes $p\ge 5$.
\begin{lem}\label{lem:minimal-p}
Let $G$ be a finite non-abelian simple group and let $p\geq 5$ be a prime dividing $|G|$. Assume that every proper subgroup of $G$ is $p$-solvable. Then one of the following holds.

\begin{itemize}
\item[$(1)$] $G=\LL_2(p)$.
\item[$(2)$] $G=\Alt_p.$
\item[$(3)$] $G=\LL_2(q)$ with $p\mid q^2-1$.
\item[$(4)$] $G=\LL_n(q), n\ge 3$ is odd, and $p$ divides $q^n-1$ but not $ \prod_{i=1}^{n-1}(q^i-1)$.
\item[$(5)$] $G=\UU_n(q), n\ge 3$ is odd, and $p$ divides $q^n-(-1)^n$ but not $ \prod_{i=1}^{n-1}(q^i-(-1)^i)$.
\item[$(6)$] $G={}^2{\rm B}_2(q)$ with $q=2^{2m+1},m\ge 1.$
\item[$(7)$] $G={}^2{\rm G}_2(q)$ and $p\mid (q^2-q+1)$, where $q=3^{2m+1},m\ge 1.$
\item[$(8)$] $G={}^2{\rm F}_4(q)$ with $q=2^{2m+1},m\ge 1$ and $p\mid (q^4-q^2+1).$
\item[$(9)$] $G={}^3{\rm D}_4(q)$ and $p\mid (q^4-q^2+1).$
\item[$(10)$] $G={\rm E}_8(q)$  and $p$ divides $ (q^{30}-1)$ but not $\prod_{i\in \{8,14,18,20,24\}}(q^i-1)$.
\item[$(11)$] $(G,p)$  is one of the following $({\rm M}_{23},23)$, $({\rm J}_{1},7 \text{ or } 19)$, $({\rm Ly},$ $37 \text{ or } 67)$, $({\rm J}_{4}, 29 \text{ or } 43)$, $({\rm Fi}_{24}',29)$, $({\rm B},47)$ or $({\rm M}, 41  \text{ or } 59 \text{ or } 71).$
\end{itemize}
\end{lem}

\begin{proof}
This is Lemma 5.4 in \cite{GT}.
\end{proof}

Let $q$ be a prime power and let $n\ge 2$ be an integer. A prime divisor $p$ of $q^n-1$ is called a primitive prime divisor or ppd of $q^n-1$ if $p$ does not divide $q^k-1$ for all integers $k$ with $1\leq k<n.$ Zsigmondy's theorem \cite{Zsig} states that such a ppd $p$ exists unless $(n,q)=(6,2)$ or $n=2$ and $q$ is a Mersenne prime. Now if $n>1$ is an integer and $p$ is a prime, then the $p$-part of $n,$ denoted by $n_p$, is the largest power of $p$ dividing $n.$ We refer the reader to \cite{KL,BHR} for the description of maximal subgroups of finite simple groups of Lie type.

\begin{prop}\label{prop:Sylow-generation} Let  $G$ be a finite non-abelian simple group and $p$ be a prime dividing $|G|$. Assume that every proper subgroup of $G$ is $p$-solvable.  Let  $P$ be a Sylow $p$-subgroup of $G$. Then $G=\la P,P^g\ra$ for some $g\in G.$
\end{prop}

\begin{proof} Let $G$ be a finite non-abelian simple group and let $p$ be a prime dividing $|G|$. Let $P\in\Syl_p(G)$. Assume that every proper subgroup of $G$ is $p$-solvable.
Let $x\in P$ with $|x|=p$. Then $\la x,x^g\ra\leq \la P,P^g\ra$ for all $g\in G$ and thus $\la x,x^g\ra$ is $p$-solvable for all $g\in G.$

(a) If $p=2$, then every finite non-abelian simple group $G$ can be generated by two Sylow $2$-subgroups by Theorem A in \cite{Guralnick}. So, we may assume that $p>2$.

(b) If $G$ is a finite non-abelian simple group of Lie type in characteristic $p$, then $G$ is generated by two Sylow $p$-subgroups by Proposition 2.5 in \cite{BrG}.

(c) Assume that  $p=3$. By Lemma \ref{lem:minimal-3}, since $3$ divides $|G|$, $G$ is a minimal simple group.  By part (b), we only need to consider the cases when $G$ is isomorphic to $ \LL_2(2^r),r$ is a prime, or $\LL_2(r),r>3$ is a prime and $5\mid r^2+1$. 

If $G\cong \LL_2(4) $, then we can check by using GAP \cite{GAP} that there exists  $g\in G$ such that $G=\la P,P^g\ra$. Assume next that $G\cong \LL_2(q)$,  $q=2^r$ or $r,$ where $r$ is an odd prime. By \cite[Theorem A*]{Guest}, there exists an element $x\in G$ of order $3$ such that $\la x,x^g\ra$ is nonsolvable for some $g\in G.$ Since $G$ is minimal simple, we must have $G=\la x,x^g\ra$ and thus $G=\la P,P^g\ra.$

(d) Assume that $p\ge 5.$ By part (b), and Lemma \ref{lem:minimal-p}, $G$ is one of the groups listed in $(2)-(11)$ in that lemma. We now consider each case in turn.

(1) Assume $G=\Alt_p$.  In this case $|P|=p$. Without loss of generality, take  $P=\la x\ra$, where $x=(1,2,\dots,p)$ is a $p$-cycle in $\Alt_p.$ Let $y=(1,2,p,p-1,p-2,\dots,3)\in \Alt_p$ be another $p$-cycle. Then $xy=(1,3,2)$ and clearly $$\la x,y\ra=\la (xy)^{-1},x\ra=\la (1,2,3),(1,2,\dots,p)\ra=\Alt_p$$ (see, for example, \cite[Theorem B]{IZ}.) Hence $G$ is generated by two Sylow $p$-subgroups.

(2) Assume $G=\LL_2(q)$ with $p\mid q^2-1$. If $q\leq 11$, we can check using GAP \cite{GAP} that the result holds. Assume $q\ge 13.$ Inspecting the argument in \cite[\S 5.1.2]{Guest}, if $x$ is any element of order $p$, then we can find $g\in G$ such that $\la x,x^g\ra\cong \LL_2(q)$ and hence $\la P,P^g\ra=G$. 

(3) Assume $G=\LL_n(q), n\ge 3$ is odd, and $p$ divides $q^n-1$ but not $ \prod_{i=1}^{n-1}(q^i-1)$. Write $q=s^f$, where $s$ is a prime and $f\ge 1$ is an integer. In this case $p$ is a ppd of $q^n-1$. Hence $P\in\Syl_p(G)$ is cyclic of order $(q^n-1)_p.$ Since $n\equiv 1$ mod $p$ and $n\ge 3$ is odd, we have  $p\ge 2n+1$ and $p\nmid n$. 

Assume that $t$ is a prime divisor of $n$ and write $n=tm$ for some integer $m\ge 1.$ Assume that $m>1$. Then $G$ has a $\mathcal{C}_3$-subgroup $H$ of type $\GL_{m}(q^t)$ (see \cite[Table 3.5A]{KL}) which is maximal and contains a Sylow $p$-subgroup of $G$. Since $n\ge 3$ is odd, $m,t\ge 3$ and so $q^t\ge 2^t\ge 8$. Therefore, $H$ is not $p$-solvable. 

Thus we can assume that $n=t$ is an odd prime. If $P$ lies in a unique maximal subgroup $H$ of $G$, then $H$ is of type $\GL_1(q^n)$ by \cite[Table B]{BBGT}. We can choose $g\in G-H$ such that $P^g\not\leq H$ and hence $G=\la P,P^g\ra$.  Assume that $P$ lies in some other maximal subgroup $M$ of $G$ not of type $\GL_1(q^n)$. As in the proof of Case 3 of Proposition 6.2 in  \cite{BBGT}, $M\in\mathcal{C}_5$ is a subfield subgroup of type $\GL_n(q_0)$, where $q=q_0^k$, $k$ is an odd prime and $(q_0^n-1)_p=(q^n-1)_p$ or $M\in\mathcal{S}$ is almost simple with socle $S\cong \LL_2(p)$ and $n=(p-1)/2$. However, in both cases, $M$ is not $p$-solvable.

(4)  Assume $G=\UU_n(q), n\ge 3$ is odd, and $p$ divides $q^n+1$ but not $ \prod_{i=1}^{n-1}(q^i-(-1)^i)$. Write $q=s^f$ where $s$ is a prime and $f\ge 1$.

(a) Assume that $n=3$. If $q=3$, then $p=7$. In this case, $|P|=7$ and $P$ lies in $\LL_2(7)$ which is not $7$-solvable.  Similarly, if $q=5$, then $p=7$ and $|P|=7$ and $P$ lies in $\Alt_7$.

Assume first that $q$ is a prime.   Let $H$ be a maximal subgroup of $G$ containing $P.$  By the proof of Proposition 6.3 in \cite{BBGT}, either $P$ lies in a unique maximal subgroup of $G$ and we are done or $P$ is contained in $\LL_2(7)$ and $p=7$; however,  $\LL_2(7)$ is not $7$-solvable.

Assume $q=s^f$ with $f>1$. In this case, if  $P$ is not contained in a unique maximal subgroup, then $P$ can be contained in a subfield subgroup of type $\GU_3(q_0)$ with $q=q_0^k$, and $k$ is an odd prime (see \cite[Table 8.5]{BHR}). However such a maximal subgroup is not $p$-solvable.

(b) Assume $n\ge 5.$ Then $p$ is a ppd of $q^{2n}-1$. Hence $p\ge 2n+1$.

Assume that $n=tm$, where $t$ is a prime divisor of $n$ and $m>1$. Since $n\ge 5$ is odd, $t,m\ge 3.$ Then $G$ has a maximal subgroup of type $\GU_m(q^t)$ and contains a Sylow $p$-subgroup of $G$. Since $q^t\ge 2^t\ge 8$, such a maximal subgroup is not $p$-solvable.

Therefore, $n=t\ge 5$ is a prime. Argue as in case (3), if $P$ lies in a unique maximal subgroup of $G$, then the conclusion holds. As in the proof of Proposition 6.4 \cite{BBGT}, $P$ lies in a subfield subgroup $H$ of type $\GU_n(q_0)$, where $q_0^k=q$ and $k\ge 3$ is a prime, or of type $\mathrm{O}_n(q)$ or $H$ is an almost simple group with socle $\LL_2(p)$ with $n=(p-1)/2$ and $7\leq p\equiv 3$ mod 4. However, in all cases, these maximal subgroups are not $p$-solvable.

(5) $G={}^2{\rm B}_2(q)$ with $q=2^{2m+1},m\ge 1.$  Then $|G|=q^2(q-1)(q+s+1)(q-s+1),$ where  $s=\sqrt{2q}=2^{m+1}$. The maximal subgroups of $G$ are listed in \cite[Table 8.16]{BHR}. Since $p$ is not the characteristic of $G$, $p>2$ and $p\mid q-1$ or $p\mid q\pm s+1$. 

Assume first that $p\mid q-1$. Then $P$ lies in maximal subgroups of the form $[q^2]:(q-1)$ and $D_{2(q-1)}$. It follows that $\la P,p^g\ra$ is solvable for all $g\in G.$  Let $x\in P$ with $|x|=p\ge 5.$ Then $\la x,x^g\ra$ is solvable for all $g\in G$. However, this is impossible in view of Theorem  A* in \cite{Guest}.

Assume that $p\mid q\pm s+1.$ In this case, $P$ lies in a maximal subgroup of the form $(q\pm s+1):4$ or a subfield subgroup of the form ${}^2{\rm B}_2(q_0)$, where $q_0^k=q,k\ge3$ is a prime and $q_0>2.$ Clearly, the subfield subgroup is not $p$-solvable (if it contains $P$). Hence $P$ lies in a unique maximal subgroup of $G$ and the result follows.

(6) $G={}^2{\rm G}_2(q)$ with $q=3^{2m+1},m\ge 1$ and $p\mid q^2-q+1$. We can use the same argument as in the previous case using \cite[Table 8.43]{BHR}.

(7) $G={}^2{\rm F}_2(q)$ with $q=2^{2m+1},m\ge 1$ and $p\mid q^4-q^2+1$.  In this case, $p$ is a ppd of $q^{12}-1$. Using the argument in  Proposition 7.2 in \cite{BBGT},  either $P$ lies in a unique maximal subgroup or it lies in a subfield subgroup ${}^2{\rm F}_2(q_0)$, which is not $p$-solvable.

(8) $G={}^3{\rm D}_4(q)$  and $p\mid q^4-q^2+1$.  We can use the argument in Proposition 7.3 in \cite{BBGT} to obtain the conclusion as in the previous case. 

(9) $G={\rm E}_8(q)$ and $p$ divides $ (q^{30}-1)$ but not $\prod_{i\in \{8,14,18,20,24\}}(q^i-1)$. In this case, $p$ is a ppd of $q^{30}-1$. From Proposition  7.10 \cite{BBGT}, either $P$ lies in a unique maximal subgroup and the result follows or $P$ can lie in a maximal exotic local subgroup $2^{5+10}\cdot \LL_5(2)$ when $|P|=p=31$ or $P$ lies in an almost simple group. In the last two possibilities, clearly, these maximal subgroups are not $p$-solvable.

(10) $(G,p)$  is one of the following $({\rm M}_{23},23)$, $({\rm J}_{1},7 \text{ or } 19)$, $({\rm Ly}, 37 \text{ or } 67)$, $({\rm J}_{4}, 29 \text{ or } 43)$, $({\rm Fi}_{24}',29)$, $({\rm B},47)$ or $({\rm M}, 41  \text{ or } 59 \text{ or } 71).$ 

By \cite[Table D]{BBGT}, $P$ lies in the unique maximal subgroup of $G$ and the result follows except for the case $(G,p)=({\rm J}_1,7)$. By the Atlas \cite{atlas}, the maximal subgroups of ${\rm J}_1$ containing a Sylow $7$-subgroup  are isomorphic to either $2^3: 7: 3$ or $7:6$. Thus $\la x,x^g\ra$ is solvable for all $g\in {\rm J}_1$, where $x\in P$ with $|x|=7$. However, this contradicts Theorem A* in \cite{Guest}. The proof is now complete. 
\end{proof}

\begin{rem}
It is conjectured in \cite{BrG} that if $G$ is a finite non-abelian simple group and if $r$ and $s$ are prime divisors of $|G|$, then $G$ can be generated by a Sylow $r$-subgroup and a Sylow $s$-subgroup. The previous proposition is just a special case of this conjecture when $r=s=p$ and $G$ is a minimal non-$p$-solvable  simple group.
\end{rem}

\begin{proof}[\bf{Proof of Theorem \ref{th5}}]
Let $G$ be a finite group and let $p$ be a prime. Let $P$ be a Sylow $p$-subgroup of $G$. If $G$ is $p$-solvable, then every subgroup of $G$ is $p$-solvable. Therefore, it suffices to show that if $\la P,P^g\ra$ is $p$-solvable for all $g\in G,$ then $G$ is $p$-solvable. Suppose not and let $G$ be a counterexample with  minimal order.
Then $\la P,P^g\ra$ is $p$-solvable for all $g\in G$ but $G$ is not solvable.

We first claim that every proper subgroup of $G$ is $p$-solvable and thus $G$ is a minimal non-$p$-solvable group. Let $H$ be a proper subgroup of $G$ and let $Q$ be a Sylow $p$-subgroup of $H$. Then $Q\leq P^t$ for some $t\in G.$ Now for every $h\in H,$ we have $$\la Q,Q^h\ra\leq \la P^t,(P^{t})^h\ra=\la P,P^{tht^{-1}}\ra ^t.$$
Since $\la P,P^{tht^{-1}}\ra$ is $p$-solvable, $\la Q,Q^h\ra$ is $p$-solvable. Therefore, by the minimality of $|G|$, $H$ is $p$-solvable.

By Proposition \ref{prop:Sylow-generation}, we know that $G$ is not a non-abelian simple group. Let $N$ be a proper nontrivial normal subgroup of $G$. Now $PN/N$ is a Sylow $p$-subgroup of $G/N$ and it satisfies the hypothesis of the theorem. Since $|G/N|<|G|$,  $G/N$ is $p$-solvable. As in the previous claim, $N$ is also $p$-solvable and thus $G$ is $p$-solvable as well. This final contradiction proves the theorem.
\end{proof}

\subsection*{Acknowledgment} The author is grateful to the referee for numerous comments
and suggestions that have significantly improved the exposition of the paper. The referee has  simplified the proofs of both Theorem \ref{th1} and Corollary \ref{cor:generation} as well as  shortened the proof of Theorem \ref{th2} significantly. The author also thanks Chris Schroeder for careful reading of several versions of the paper.

\end{document}